\documentclass{amsart}
\usepackage{graphicx} % Required for inserting images
\usepackage{amsmath}
\usepackage{amsfonts}
\usepackage{amssymb}
\usepackage{amsthm}
\usepackage{mathtools} % added to make \psmallmatrix work.
\usepackage[capitalize]{cleveref}
\usepackage{fullpage}
\usepackage{amsrefs}
\BibSpec{misc}{
+{}{\PrintAuthors}   {author}
+{,}{ \textit}        {title}
+{,}{ }              {publisher}
+{}{ \parenthesize}                  {date}
+{, }{}              {edition}
+{.}{}              {transition}
+{ url: }{}         {url}
+{}{}               {review}
}

\usepackage{float}

\newcommand{\pmx}[4]{\begin{pmatrix} #1 & #2 \\ #3 & #4 \end{pmatrix}}

\newcommand{\mtwo}{M_2(\mathbb{N}_0)^\bullet}
\newcommand{\B}{\mathcal{B}}
\newcommand{\Bstar}{\mathcal{B}^*}
\newcommand{\absdet}[1]{\left|\det\left( #1 \right)\right|}

\newtheorem{theorem}{Theorem}[section]
\newtheorem*{theorem*}{Theorem}
\newtheorem{lemma}[theorem]{Lemma}
\newtheorem{prop}[theorem]{Proposition}
\newtheorem{cor}[theorem]{Corollary}
\newtheorem*{cor*}{Corollary}

\title{Atoms in the Semigroup of Non-Negative Integer Matrices}
\author{Lindsay Dever}
\address[L. Dever]{Department of Mathematics\\Millersville University\\Millersville, PA 17551, USA} 
\email{Lindsay.Dever@millersville.edu}
\author{Eva G. Goedhart}
\address[E.~G. Goedhart]{Department of Mathematics\\Bryn Mawr College\\Bryn Mawr, PA 19010, USA} 
\email{egoedhart@brynmawr.edu}
\author{Gregory S. Heilbrunn}
\address[G. Heilbrunn]{Independent Researcher\\Wheatley Heights, NY 11798, USA} 
\email{gregoryheilbrunn@gmail.com}
\author{Tony~W.~H.~Wong}
\address[T.~W.~H.~Wong]{Department of Mathematics\\Kutztown University of Pennsylvania\\Kutztown, PA 19530, USA} 
\email{wong@kutztown.edu}

 
\date{}

\begin{document}
\begin{abstract}
In the semigroup $\mtwo$, two-by-two matrices with non-negative integer entries and non-zero determinant, we study the factorization of matrices into atoms, or irreducible matrices. In 2022, Baeth et al.\  listed some fundamental classes of atoms in $\mtwo$; however, the factorability of most matrices in $\mtwo$ remains unknown. We identify two additional classes of atoms: a class of atoms with determinant $p$, $2p$, or $4p$, for  $p$ prime, and a class of atoms in which the main diagonal is much ``larger" than the off-diagonal (or vice versa).
Finally, we show that bisymmetric matrices with relatively prime  entries are a divisor-closed subset of $\mtwo$ and use a factor search algorithm to classify bisymmetric atoms of $\mtwo$ with minimum entry up to 4000.

\end{abstract}

\maketitle

\section{Introduction}

In 1963, Cohn~\cite{Cohn63} extended the concept of unique factorization to a non-commutative setting. In particular, semigroups of matrices offer a setting where factorization is not necessarily unique. For example, Jacobsen and Wisner~\cite{JW86} found that factorization is non-unique in the semigroup of two-by-two positive integral matrices with determinant one and is unique in the semigroup of two-by-two non-negative integral matrices with determinant one.

While a semigroup does not have prime elements in a traditional algebraic sense, we can study the {\it atoms}, or irreducibles, which do not factor  unless one factor is a unit. (See Section \ref{sect: background} for detailed definitions.) In early studies of two-by-two  matrix semigroups with positive integer entries, Chuan and Chuan~\cite{ ChuanChuan86, ChuanChuan1985} classified all atoms in the semigroup of determinant one matrices, and they found all atoms with prime determinant in the semigroup with positive determinant.

In 2006, Geroldinger and Halter-Koch~\cite{GHK06} summarized the modern theory of non-unique factorization. Baeth et al.~\cite{Baeth11, Baeth21} applied many of these concepts to the semigroup of nonnegative integer matrices with non-zero determinant, $M_n(\mathbb{N}_0)^\bullet$.  Specifically, from~\cite{Baeth21} we know that factorization in $\mtwo$ is non-unique; in addition, no atom is {\it prime-like} in the sense that an atom may appear with different multiplicities in factorizations of the same matrix. A complete classification of the atoms of $\mtwo$ is unknown; however, Baeth et al.~\cite{Baeth21} found the following classes of atoms in $\mtwo$.

\begin{theorem}\label{TheoremBaethAtoms}
\cite[Theorem 4.6, Prop. 4.8]{Baeth21}
The following are atoms of $\mtwo$:
    \begin{enumerate}
        \item $\pmx{1}{1}{1}{0}, \pmx{1}{1}{0}{1}, \pmx{1}{0}{1}{1}, \pmx{0}{1}{1}{1}$,
        \item $\pmx{p}{0}{0}{1}, \pmx{0}{p}{1}{0}, \pmx{0}{1}{p}{0}. \pmx{1}{0}{0}{p}$, for $p$ prime,
        \item $\pmx{w}{x}{y}{z}, \pmx{x}{w}{z}{y}, \pmx{x}{z}{w}{y}, \pmx{z}{x}{y}{w}$, with $w \in \{1,2,3\}$, $\gcd{(wz, xy)}=1$, and $w\leq z<x,y$,
        \item $\pmx{x}{x+1}{x+1}{x},\pmx{x+1}{x}{x}{x+1}$, where $x\in \mathbb{N}$ and $2x+1$  is prime.
    \end{enumerate}
\end{theorem}

In Proposition \ref{PropBaethAtoms}, Baeth et al.~\cite{Baeth21} provides additional criteria for us to find other atoms of the semigroup. In particular, all other atoms not described in Theorem \ref{TheoremBaethAtoms} must have {\it relatively prime adjacent entries} and must be {\it doubly-balanced}, meaning not reducible by row or column operations. 

We found additional classes of atoms in $\mtwo$ within those criteria, the first of which we found by investigating determinants. We show in Theorems~\ref{thm:absdetp} and \ref{thm:absdet32} that these restrictions imply matrices with certain determinants are atoms. In particular, a doubly-balanced matrix $X$ with relatively prime adjacent entries is an atom when $$\absdet{X}\in \{16,32\}\cup \{p,2p,4p \mid p \text{ prime}\}.$$

We also specialize to the set of bisymmetric matrices and search for bisymmetric atoms in the larger semigroup of $\mtwo$. If there is a common factor between the entries of a bisymmetric matrix, then the matrix will factor trivially; therefore we consider the subset of bisymmetric matrices with relatively prime adjacent entries, which we denote by
\[\B^*=\left\{\pmx{x}{y}{y}{x}\in \mtwo:\gcd(x,y)\right\}.\]

The set of bisymmetric matrices is not divisor-closed; that is, the factors of bisymmetric matrices are not necessarily bisymmetric. However, when we consider the subset of bisymmetric matrices with relatively prime adjacent entries, we prove that the factors must also be bisymmetric.
\newtheorem*{thm51}{\cref{thm:bisym factors}}
\begin{thm51} Let $X=\pmx{x}{y}{y}{x}\in \Bstar$. If $X$ factors into $X=AB$, where $A,B \in \mtwo$, then $A,B \in \Bstar$. 
\end{thm51}
\noindent This shows that the set $\B^*$ is divisor-closed in $\mtwo$. Note that it is a necessary condition that $X$ has relatively prime adjacent entries; see equation \eqref{eq1factor} for an example of a bisymmetric matrix that factors into non-bisymmetric factors.

This allows us to consider a smaller subset of possible factors when determining if a bisymmetric matrix is an atom. We use an algorithm to search potential factors and classify all bisymmetric atoms in $\mtwo$ with minimum up to $4000$. Critically, the assumption that factors must be bisymmetric greatly simplifies the search for factors. For a fixed minimum value, we give lists of factorizations found; all other bisymmetric matrices with this minimum are atoms.

\newtheorem*{thm71}{\cref{thm: algorithm}}
\begin{thm71}
Let $X=\pmx{x}{y}{y}{x} \in \Bstar$ with $\min\{x,y\}\leq 4000$. The matrices listed in \cite{BisymData26} (and their associates) are not atoms in $\mtwo$. If $X$ and its associate matrix $\pmx{y}{x}{x}{y}$ do not appear in \cite{BisymData26}, then $X$ is an atom.
\end{thm71}

In \cref{sec: upper bound}, we discover an additional class of atoms using bounds on the diagonal or off-diagonal entries. In particular, in \cref{wxyzupperbound2},  we show that for a doubly-balanced, factorable matrix $X = \pmx{w}{x}{y}{z}$ satisfying $\gcd(wz, xy)=1$, then 
$$x,y\le2(z-1)(w-1)\qquad\text{ and}\qquad
w, z \leq 2(x-1)(y-1).$$
This leads immediately to a new class of atoms where this bound is not satisfied.
\newtheorem*{cor1}{\cref{cor: upper bound atoms}}
\begin{cor1}
    Let $X=\pmx{w}{x}{y}{z}\in \mtwo$ such that $w\le z<x,y$ and $\gcd{(wz, xy)}=1.$ If $$x>2(z-1)(w-1) \text{ \hspace{2pc} or \hspace{2pc} } y>2(z-1)(w-1),$$ then $X$ and its associates are atoms in $\mtwo$.
\end{cor1}

For bisymmetric matrices, we prove a stronger upper bound on factorable matrices in \cref{cor: upperboundBisym}, which we use to find another class of atoms.

\newtheorem*{cor64}{\cref{thm:boundBisym}}
\begin{cor64} Let $X=\pmx{x}{y}{y}{x} \in \Bstar$. If $\max\{x,y\}> 1+ \frac{(\min\{x,y\})^2}{4}$, then $X$ is an atom in $\mtwo$.
\end{cor64}

\cref{thm:bisym factors} has further implications for discovering additional atoms in $\mtwo$. For example, results of Pomonarenko allow us to conclude in \cref{sec: divisor closed} that matrices in the set
$$\left \{\pmx{2^n+1}{2^n-1}{2^n-1}{2^n+1}:n\in \mathbb{N}_0\right\}$$
are atoms in the semigroup of non-negative integer matrices \cite[Theorem 5]{Ponom22}.

\section{Background and definitions}\label{sect: background}

We start off with some well-known definitions that we include for convenience. A {\it semigroup} is a set $S$ which is closed under an associative operation, $*$. A semigroup is {\it cancellative} if for all $x, y, z \in S$, $x*y=x*z$ implies that $y=z$ and similarly $y*x=z*x$ implies that $y=z$. A {\it unit} is an element $u\in S$ so that $u*v=v*u=1_S$, the multiplicative identity of $S$. In a semigroup, our primary interest is in the factorization of elements within that semigroup. An  element $a\in S$ is an {\it atom} (or an irreducible element of $S$) if $a=xy$ implies that $x$ or $y$ is a unit. A semigroup $S$ is said to be {\it atomic} if each non-unit of $S$ factors into a product of atoms.

The set $\mtwo$ under matrix multiplication is a cancellative semigroup with identity. Further, Baeth et al. showed that $\mtwo$ is atomic and satisfies the finite factorization property, in other words, that there are finitely many ways to factor a given element  \cite[Theorem 5.1]{Baeth21}. However, factorization in $\mtwo$ is not unique; for example,
\begin{equation}\label{eq1factor}
\pmx{15}{10}{10}{15}=\pmx{1}{2}{3}{1}\pmx{1}{4}{7}{3}= \pmx{5}{0}{0}{5} \pmx{3}{2}{2}{3}.
\end{equation}

The only two units in $\mtwo$ are the identity, $I=\pmx1001$, and the exchange matrix, $J=\pmx0110$. Given a matrix $X=\pmx{w}{x}{y}{z}$, the \textit{associates} of $X$ are the matrices obtained by multiplying $X$ by a unit; namely, the associates of $X$ are $X$ itself, 
$$JX=\pmx{y}{z}{w}{x}, \qquad XJ=\pmx{x}{w}{z}{y},\text{ and}\qquad JXJ=\pmx{z}{y}{x}{w}.$$
We will often refer to factorization ``up to associates'' by which we mean that there are additional factorizations which include the associate matrices of the factors presented.

Frequently, we describe a matrix as having \textit{relatively prime adjacent entries}. This indicates that the matrix meets the condition that the entries in each row are relatively prime and the entries in each column are relatively prime. This condition is equivalent to requiring that $\gcd(xy, wz)=1$ for a matrix $X=\pmx{w}{x}{y}{z}$.

One way to obtain atoms in a semigroup is to consider factorization in a smaller subset. A subset $T$ is {\it divisor-closed } in a semigroup $S$ if when $X\in T$ factors into $X=YZ$ , then the factors $Y$ and $Z$ are in $T$. If $T$ is a divisor-closed semigroup and $A$ is an atom of $T$, then $A$ is an atom of $S$ \cite[Prop. 3.2]{Baeth21}. 

One such example is the semigroup 
\[T = \left\{\pmx{x}{x+1}{x+1}{x},\pmx{x+1}{x}{x}{x+1}: x\in \mathbb{N}_0\right\}\]
which was shown to be divisor-closed in $\mtwo$ \cite[Prop. 3.10]{Baeth21}. Therefore, any atoms in $T$ are also atoms in the full semigroup, $\mtwo$. However, not all elements of this subsemigroup are atoms; for example,
\[\pmx{5}{4}{4}{5}=\pmx1221 \pmx1221\]
is not an atom in $T$ or $\mtwo$.

In this paper, we  consider the subsemigroup $\B$ of bisymmetric matrices,
\[\B= \left\{\pmx{x}{y}{y}{x}: x,y \in \mathbb{N}_0, x\neq y\right\}.\]
This is a commutative subsemigroup which is closed under multiplication; however, it is not divisor-closed in $\mtwo$. See \eqref{eq1factor}
for an example of a bisymmetric matrix which factors into matrices that are not bisymmetric. However, we specify to the subset $\B^*$ of matrices with relatively prime adjacent entries and prove that this subset is divisor-closed in \cref{thm:bisym factors}. However, note that $\B^*$ is not a subsemigroup as it is not closed under multiplication; for example,
\[\pmx{3}{1}{1}{3} \pmx{1}{5}{5}{1}=\pmx{8}{16}{16}{8}\]
is not in $\B^*$, although the factors are in $\B^*$.

%*************************************************************

\section{Preliminary Results}

We  now consider various conditions that must be met for a matrix to be an atom. We borrow some terminology from Raney \cite{Raney73}. We  say that a row of a matrix is dominant if both entries are greater than or equal to the corresponding entries in the other row. Similarly, a column is dominant if both entries are greater than or equal to the corresponding entries in the other column. If neither row is dominant, then we say that a matrix is row-balanced, and if neither column is dominant, then we say that the matrix is column-balanced. A matrix is doubly-balanced if it is both row- and column-balanced.

If a matrix is row or column dominant, then it can be factored by an elementary row or column operation -- in particular, by an associate of $\pmx{1}{1}{0}{1}$. Therefore, atoms of $\mtwo$ must necessarily be doubly-balanced, unless they are associates of $\pmx{1}{1}{0}{1}$, which was shown in \cite{Baeth21}.

\begin{prop}\cite[Prop. 4.5]{Baeth21} \label{PropBaethAtoms}
Suppose that $X = \pmx{w}{x}{y}{z}$ is an atom in $\mtwo$. Then one of the following holds:
\begin{itemize}
    \item If $X$ has at least one zero entry, then $X$ is equal to one of the matrices listed in (1) or (2) of Theorem \ref{TheoremBaethAtoms}.
    \item If $X$ has no zero entries, then, multiplying by $\pmx{0}{1}{1}{0}$ to obtain the minimum entry in the upper left, $w \leq z< x,y $ and $\gcd(x, wz)=\gcd(y,wz)=1$.
\end{itemize}
\end{prop}
The second condition is equivalent to requiring that the matrix is doubly-balanced with $\gcd(xy, wz)=1$.

We also obtain that for a doubly-balanced matrix that factors, the first factor must be row-balanced, and the second factor must be column-balanced.

\begin{prop}\label{lem:balanced} Suppose that a doubly-balanced matrix $X\in \mtwo$ factors into $X=AB$, where $A, B \in \mtwo$. The first factor $A$ is row-balanced, and the second factor $B$ is column-balanced.
\end{prop}

\begin{proof}
   Suppose for contradiction's sake that $A=\pmx{a}{b}{c}{d}$ is not row-balanced. Without loss of generality, suppose $a\geq c$ and $b\geq d$. Then, $A$ will factor into
    \[A= \pmx{1}{1}{0}{1}\pmx{a-c}{b-d}{c}{d}\]
    where both matrices are non-negative. Letting $R=\pmx{1}{1}{0}{1}$ and $A'=\pmx{a-c}{b-d}{c}{d}$, we see that  $A=RA'$ implies  $X=AB = R(A' B)$ is not row-balanced.

    Similarly, if $B$ is not column-balanced, then $B$ will factor into $B=B' R_2$ where $R_2=\pmx1011$ or $\pmx1101$. But then $X=AB=(AB')R_2$ is not column-balanced.
\end{proof}

If we additionally require that adjacent entries are relatively prime, then we also obtain that the factors must not have any zero entries.

\begin{lemma}\label{LemmaNoZeroes} Let $X=\pmx{w}{x}{y}{z}\in \mtwo$ be doubly-balanced and $\gcd{(wz, xy)}=1.$ If $X=AB$ where $A, B \in \mtwo$ are not units, then $A$ and $B$ have no zero entries.
\end{lemma}
\begin{proof}
    Suppose that $X$ factors into $X=\pmx{a}{b}{c}{d}\pmx{e}{f}{g}{h}$ and $a=0$. Then,
    \[X=\pmx{0}{b}{c}{d}\pmx{e}{f}{g}{h}=\pmx{bg}{bh}{ce+dg}{cf+dh}.\]
    If $b>1$, then this matrix would not satisfy the condition that $\gcd(wz,xy)=1$. If $b=0$, the determinant of the first factor would be zero. Therefore, we must have $b=1$. Then $X$ simplifies to 
    \[X=\pmx{0}{1}{c}{d}\pmx{e}{f}{g}{h}=\pmx{g}{h}{ce+dg}{cf+dh}.\]
    Note that $X$ is not row-balanced unless $d=0$; otherwise $g\leq ce+dg$ and $h\leq cf+dh$. Therefore, $d=0$ and $X$ simplifies to
    \[X=\pmx{0}{1}{c}{0}\pmx{e}{f}{g}{h}=\pmx{g}{h}{ce}{cf}.\]
    However, this contradicts the condition that $\gcd(wz, xy)=1$ unless $c=1$, in which case the first factor is the exchange matrix.

    If  any entry in either factor is zero, we can show by similar argument that the factor is either the exchange or identity matrix.
\end{proof}

If we factor a doubly-balanced matrix with relatively prime adjacent entries into atoms, we additionally obtain that the first and last atoms in the factorization have non-zero entries, as opposed to the atoms listed in (1) and (2) from \cref{TheoremBaethAtoms}. This also results in the first and last atoms in the factorization being of similar character: doubly-balanced with relatively prime adjacent entries.

\begin{cor}\label{lem:bookends}
Let $X\in \mtwo$ be factorable, doubly-balanced, and have relatively prime adjacent entries. For all factorizations $X=X_0\cdots X_n$ into $X_i$ atoms, it follows that $X_0$ and $X_n$ have no zero entries, are doubly-balanced, and have relatively prime adjacent entries.
\end{cor}

\begin{proof}
Let $X \in \mtwo$ be a factorable, doubly-balanced matrix with relatively prime adjacent entries. Suppose $X$ factors into $X=X_0\cdots X_n$ where each $X_i$ is an atom. We have factorizations into $X = X_0 (X_1 \cdots X_n)$ or $X = (X_0 \cdots X_{n-1}) X_n$, and by \cref{LemmaNoZeroes}, then $X_0$ and $X_n$ must have no zero entries. (Note that $X_0$ and $X_n$ are atoms and so are not units.) Therefore by Proposition 3.1, $X_0$ and $X_n$ must be doubly-balanced with relatively prime adjacent entries.
\end{proof}

It immediately follows that a factorable, doubly-balanced matrix with relatively prime adjacent entries must factor into at least two matrices of similar character. If it does not, then we conclude it must be an atom.

\section{Atoms in $\mtwo$ Based on Determinants}\label{Sec. Determinant}

We now turn our attention to the determinants of matrices. 
Throughout this section, we focus on matrices in $\mtwo$ which are doubly-balanced with relatively-prime adjacent entries, since these matrices are the ``candidates'' to be atoms from \cref{PropBaethAtoms}. First, we will obtain a lower bound on the determinant from the minimum entry of the matrix.

\begin{lemma}\label{lem:DetMin}
    Suppose $X\in \mtwo$ is a doubly-balanced matrix with relatively prime adjacent entries, and let $m$ be the minimum entry of $X$. Then, $\absdet{X}\ge 2m+1$.
\end{lemma}

\begin{proof}
    Let $X$ be doubly-balanced with relatively prime adjacent entries. Multiply $X$ by the exchange matrix $J$ until the matrix is in the form $\pmx{w}{x}{y}{z}$ where $w$ is its minimum entry and by \cref{PropBaethAtoms}, $1\le w\le z<x,y$. Without loss of generality, we will call this reduced form $X$. Note that the absolute value of the determinant of $X$ remains unchanged since $\absdet{J}=1$. It follows that $x,y\ge z+1$ and $xy \ge (z+1)(z+1).$ Since $z\ge w$, we find that
    \begin{align*}
        xy &\ge (w+1)(z+1) \\
        &= wz+w+z+1 \\
        &\ge wz+2w+1.
    \end{align*}
Subtracting $wz$ from the final inequality results in $xy - wz \ge 2w+1$ resulting in the absolute value of the determinant of $X$, $$\absdet{X}=|wz-xy| = xy-wz\ge 2w+1,$$ where $w$ is the minimum entry of $X$.
\end{proof}

Since the entries of matrices in $\mtwo$ are non-negative, we can use the lower bound in Lemma \ref{lem:DetMin}, along with divisibility properties, to get a restriction on the small determinants $1$, $2$, and $4$.

\begin{lemma} \label{lem:Det124}
    Suppose $X\in \mtwo$ is a doubly-balanced matrix with relatively prime adjacent entries. Then $\absdet{X}\not \in \{1,2,4\}.$ 
\end{lemma}
\begin{proof}
     Let $X\in \mtwo$ be a doubly-balanced matrix with relatively prime adjacent entries. By \cref{lem:DetMin}, $\absdet{X} \ge 2m+1$ where $m$ is the minimum entry of $X$. No entry can be zero, so $m \ge  1$ implies $\absdet{X}\ge 3,$ and $\absdet{X} \not \in \{1,2\}.$

     Now suppose for sake of contradiction that $\absdet{X}=4$. \cref{lem:DetMin} bounds $2m+1\le 4,$ which implies $m=1$. Let $X$ have arbitrary entries $X=\pmx{w}{x}{y}{z}$,  and notice how $\absdet{X}=|wz-xy|=4$. Since their difference is even, $wz$ and $xy$ will have the same parity, and we find that they must both be odd; otherwise, an adjacent pair of entries in $X$ would share a factor of $2$ and not be relatively prime. Since $X$ is doubly-balanced, there are two cases to consider. In case 1, we let $w>x$ and $z>y$, which implies $w\geq x+1$ and $z\geq y+1$ since each entry is an integer.  Then, $wz \ge (x+1)(y+1)$ implies that
     \begin{align*}
     wz=xy+4 &\ge (x+1)(y+1)\\
     &= xy+x+y+1.
     \end{align*}
      This simplifies to $3\ge x+y$, which implies that either $xy=2$, which is even, and thus a contradiction, or $xy=1$. If $xy=1$, then $|wz-xy|=4$ implies that $wz=5$ (since $w$ and $z$ are non-negative), so the remaining entries are $1$ and $5$, contradicting that the matrix is doubly-balanced. 
     
     In case 2, we flip the inequalities so $w<x$ and $z<y$, and hence $xy \ge (w+1)(z+1)$. We arrive at a similar inequality,
     \begin{align*}
     xy=wz+4 &\ge (w+1)(z+1)\\
     &= wz+w+z+1,
     \end{align*}
     which simplifies to $3\ge w+z$. Employing the same techniques from case 1, we find that $wz=1$ or $2$, which both give contradictions. 
     
     Both cases being exhausted, it follows that $\absdet{X}\not \neq {4}$.
\end{proof}

Once we eliminate $1$, $2$, and $4$ from possible determinants of doubly-balanced matrices with relatively prime adjacent entries, we can also obtain restrictions on the determinants on factorable matrices, since determinants are multiplicative. In particular, the determinant of a factorable matrix of this type cannot be $p$, $2p$, or $4p$, where $p$ is prime.

\begin{theorem}\label{thm:absdetp}
    Suppose $X\in \mtwo$ is a doubly-balanced matrix with relatively prime adjacent entries. The matrix $X$ is an atom if the absolute value of the determinant is $p$, $2p$, or $4p$, where $p$ is prime.
\end{theorem}

\begin{proof}
    Let $X\in \mtwo$ be a doubly-balanced matrix with relatively prime adjacent entries where $\absdet{X}=mp$ for $p$ prime and $m\in\{1,2,4\}$. Assume, by way of contradiction, that $X$ is not an atom. By \cref{lem:bookends}, we know that $X$ can be expressed as a product of atoms $X_{0}\cdots X_{n}$ where $X_0$ and $X_n$ are also doubly-balanced with relatively prime adjacent entries.  Applying Lemmas 4.1 and 4.2, this means that $d_0=\absdet {X_0}$ and $d_n=\absdet{X_n}$ must be positive with $d_0,d_n\not \in \{1,2,4\}.$ We show that for all three values of $m$, the atoms $X_0$ and $X_n$ cannot exist. 
    
    If $m=1$, then $\absdet{X}=p$, $$d_0d_n \le p \qquad \text{ and} \qquad d_0,d_n\mid p.$$
    This implies $d_0\mid p$ and $d_n\mid p$. Neither can be $1$, so $d_0=d_n=p$ and thus, we arrive at a contradiction where $d_0 d_n=p^2.$
    
     If instead $m=2$, then $$d_0d_n \le 2p \qquad \text{ and} \qquad d_0,d_n\mid 2p.$$
    Since neither $d_0$ nor $d_n$ equals $1$ or $2$, this again implies they both divide $p$, which leads to contradiction.

     Finally, if $m=4$, then $$d_0d_n \le 4p \qquad \text{ and} \qquad d_0,d_n\mid 4p.$$
    Since $d_0$ cannot equal $1$, $2$, or $4$, it cannot divide $4$. 
    Therefore, $d_0, d_n\mid 4p$ implies that $d_0$ is equal to $p$, $2p$, or $4p$. But that implies that $d_n$ is equal to $4$, $2$, or $1$, which is a contradiction.
\end{proof}

We also apply Lemma 4.2 to find more determinants that yield additional atoms.

\begin{theorem} \label{thm:absdet32}
        Suppose $X\in \mtwo$ is a doubly-balanced matrix with relatively prime adjacent entries. The matrix $X$ is an atom if the absolute value of its determinant is $16$ or $32$.
\end{theorem}

\begin{proof}
    Let $X\in \mtwo$ be a doubly-balanced matrix with relatively prime adjacent entries where $\absdet{X}=32$. If $X$ is not an atom, then by \cref{lem:bookends}, $X$ factors into $X = X_0\cdots X_n$, where $X_i$ are atoms and $X_0$ and $X_n$ are doubly-balanced with relatively prime adjacent entries. By the multiplicative property of determinants,
    \[\absdet X = \prod_{i=0}^n \absdet{X_i}=32=2^5\]
    Therefore, we must have that $\absdet{X_0}=2^a$ and $\absdet{X_n}=2^b$ for non-negative integers $a$ and $b$. However, by \cref{lem:Det124}, $\absdet{X_0}$ and $\absdet{X_n}$ cannot equal $1$, $2$, or $4$, so they must be at least $8$. But then 
    $\absdet{X} \geq \absdet{X_0} \absdet{X_n} \geq 64$ contradicts that $\absdet{X}=32$. Therefore $X$ is an atom.

    The proof that $|\det (X)|=16$ implies that $X$ is an atom is identical.

\end{proof}

%%%%%%%%%%%%%%%%%%%%%%%%%%%%%%%%%%%%%%%%%%%%%%%%%%%%%%

\section{Divisor Closed Subset of Bisymmetric Matrices}\label{sec: divisor closed}

Recall the subset of bisymmetric matrices with relatively prime adjacent entries,
\[\B^*=\left\{\pmx{x}{y}{y}{x}\in \mtwo:\gcd(x,y)=1\right\}.\]
We will show that this subset is divisor-closed in $\mtwo$.

\begin{theorem}\label{thm:bisym factors} Let $X=\pmx{x}{y}{y}{x}\in \Bstar$. If $X$ factors into $X=AB$, where $A,B \in \mtwo$, then $A,B \in \Bstar$. 
\end{theorem}

\begin{proof}
Let $A=\pmx{a}{b}{c}{d}$ and $B=\pmx{e}{f}{g}{h}$. By multiplying $A$ and $B$ and equating the entries to those in $X$, we get
\begin{equation}\label{eqn:sys}
\begin{cases}
ae+bg=cf+dh,\\
af+bh=ce+dg.
\end{cases}
\end{equation}
If $a=c$, then by adding up the equations in \eqref{eqn:sys}, we have $b(g+h)=d(g+h)$. This implies that $b=d$, contradicting that the determinant of $A$ is nonzero. Hence, $a\neq c$. Assume without loss of generality that $a>c$; otherwise, we may switch the rows in $A$ and $X$ by multiplying by the exchange matrix, $J$. 

By treating \eqref{eqn:sys} as a system of linear equations, we can solve for the variables $e$ and $f$ as
\begin{equation}\label{eqn:ef}
\begin{cases}
e=\dfrac{1}{a^2-c^2}\big((ad-bc)h+(cd-ab)g\big),\\
f=\dfrac{1}{a^2-c^2}\big((ad-bc)g+(cd-ab)h\big),
\end{cases}
\end{equation}
and we obtain
\begin{equation}\label{eqn:xy}
\begin{cases}
x=ae+bg=\dfrac{1}{a^2-c^2}(ad-bc)(ah+cg),\\
y=af+bh=\dfrac{1}{a^2-c^2}(ad-bc)(ag+ch).\\
\end{cases}
\end{equation}
Note that $ad-bc>0$ since $x$ and $y$ are nonnegative with at least one of them positive.

Since $\gcd(x,y)=1$, we have from \eqref{eqn:xy} that
\begin{equation}\label{eqn:a^2-c^2first}
a^2-c^2=(ad-bc)\gcd(ah+cg,ag+ch).
\end{equation}
In particular, $(ad-bc)\gcd(g,h)$ divides $a^2-c^2$. Since $(ad-bc)\gcd(g,h)$ divides both $(ad-bc)h$ and $(ad-bc)g$, we deduce from \eqref{eqn:ef} that $(ad-bc)\gcd(g,h)$ divides both $(cd-ab)g$ and $(cd-ab)h$. This implies that $(ad-bc)\gcd(g,h)$ divides $(cd-ab)\gcd(g,h)$, or equivalently $ad-bc$ divides $cd-ab$. Consequently, either $cd-ab=0$, $cd-ab\leq-(ad-bc)$, or $cd-ab\geq ad-bc$.

If $cd-ab\leq-(ad-bc)$, then since both $e$ and $f$ are nonnegative, we have $g=h$ and thus $e=f$, contradicting that the determinant of $B$ is nonzero. If $cd-ab\geq ad-bc$, then $c(b+d)\geq a(b+d)$, contradicting that $a>c$. Hence, $cd-ab=0$, implying that $a^2-c^2$ divides $(ad-bc)\gcd(g,h)$ since both $e$ and $f$ are integers. As a result,
\begin{equation}\label{eqn:a^2-c^2second}
a^2-c^2=(ad-bc)\gcd(g,h).
\end{equation}

From $cd-ab=0$, we have $d=\frac{ab}{c}$, so $a^2-c^2=\big(\frac{a^2b}{c}-bc\big)\gcd(g,h)=(a^2-c^2)\frac{b}{c}\gcd(g,h)$. This implies that $c=b\gcd(g,h)$. Combining with $d=\frac{ab}{c}$, we have $a=d\gcd(g,h)$. Lastly, if we compare \eqref{eqn:a^2-c^2first} and \eqref{eqn:a^2-c^2second}, we get $\gcd(a,c)=1$, so $\gcd(g,h)=1$. Therefore, $a=d$ and $b=c$, and we also have $e=h$ and $f=g$ from \eqref{eqn:ef}, which leads to the conclusion that $A,B\in\Bstar$.
\end{proof}

This result is surprising because bisymmetric matrices are not generally divisor-closed, unless we require that the entries are relatively prime. For example, consider the matrix $\pmx{15}{10}{10}{15}$. This has the non-trivial factorization
\begin{equation*}\label{ex1factor}
\pmx{15}{10}{10}{15}=\pmx{1}{2}{3}{1}\pmx{1}{4}{7}{3},
\end{equation*}
where both factors are atoms by \cref{TheoremBaethAtoms}, doubly-balanced, have no zero entries, and have relatively-prime adjacent entries. However, when the entries of the bisymmetric matrix are relatively prime, this type of factorization is impossible.

Note that $\B^*$ is not a semigroup, since it is not closed under multiplication. However, we have shown that $\B^*$ is a divisor-closed subset of $\mtwo$. Therefore, bisymmetric matrices that factor in $\B^*$ will also factor in $\mtwo$, and matrices that do not factor in $\B^*$ will not factor in $\mtwo$. This allows us to use the relatively simple subset to determine if a matrix factors. We will demonstrate this with an example. 

The group of bisymmetric matrices is isometric to the group of ``perplex numbers'', also called split complex numbers, hyperbolic numbers, or various other names. Pomonarenko proposed the study of ``perplex integers.'' There are two possible definitions but the first, $\mathbb{P}_1$, is isomorphic to the semigroup of bisymmetric integer matrices. Although we are studying the semigroup of bisymmetric matrices with non-negative entries, matrices that do not factor into integer bisymmetric matrices will then not factor into non-negative bisymmetric matrices.

Pomonarenko found a class of  irreducible perplex integers that correspond to bisymmetric integer matrices which do not factor \cite[Theorem 5]{Ponom22}, namely
$$\left\{\pmx{2^n+1}{2^n-1}{2^n-1}{2^n+1}:n\in \mathbb{N}_0\right\}.$$
Note that these entries are relatively prime since they are odd and differ by 2. By \cref{thm:bisym factors}, since any factors must be bisymmetric, these matrices must be atoms in $\mtwo$.

%%%%%%%%%%%%%%%%%%%%%%%%%%%%%%%%%%%%%%%%%%%%%%%%%%%%%%%%%%
\section{Upper Bound on Factorable Matrices in $\mtwo$}\label{sec: upper bound}

We obtain an additional class of atoms by finding an upper bound on factorable matrices in $\mtwo$. Then we will tighten this upper bound for bisymmetric matrices.

\begin{theorem}\label{wxyzupperbound2} Let $X=\pmx{w}{x}{y}{z}\in \mtwo$ such that 
$X$ is doubly balanced and $\gcd(wz, xy)=1.$ If $X$ is not an atom then it must satisfy
$$x,y\le2(z-1)(w-1),$$
$$w, z \leq 2(x-1)(y-1).$$
\end{theorem}
\begin{proof}
    Let $X=\pmx{w}{x}{y}{z}\in \mtwo$ be doubly-balanced with $\gcd(wz, xy)=1.$ Suppose $X=\pmx{a}{b}{c}{d}\pmx{e}{f}{g}{h}.$ Then observe the equations
    \begin{align}
        w&=ae+bg \label{wxyzupperbound|w=}\\
        x&=af+bh \label{wxyzupperbound|x=} \\
        y&=ce+dg \label{wxyzupperbound|y=} \\
        z&=cf+dh \label{wxyzupperbound|z=}.   
    \end{align}

We can use the fact that from Lemma \ref{LemmaNoZeroes} all variables are greater than zero  to show that
$$a\le ae=w-bg \le w-1$$ using \cref{wxyzupperbound|w=}.  The same equation shows that $b,g,e\le w-1$ by similar logic. Similarly, \cref{wxyzupperbound|z=} shows that $c,d,f,h\le z-1$. We can apply these inequalities to \cref{wxyzupperbound|x=,wxyzupperbound|y=} to show 
   $$x=af+bh \le (w-1)(z-1)+(w-1)(z-1) \le 2(w-1)(z-1)$$
and 
$$y=ce+dg \le (w-1)(z-1)+(w-1)(z-1) \le 2(w-1)(z-1).$$
\end{proof}

One consequence of \cref{wxyzupperbound2} is that any doubly-balanced matrix with relatively prime adjacent entries which does not satisfy the above bounds is necessarily an atom. This gives an entire class of atoms; if we fix, for example, $w$ and $z$, then we can generate atoms by making $x$ or $y$ sufficiently large.

\begin{cor}\label{cor: upper bound atoms}
    Let $X=\pmx{w}{x}{y}{z}\in \mtwo$ such that $w\le z<x,y$ and $\gcd(wz, xy)=1.$ If $x>2(z-1)(w-1)$ or $y>2(z-1)(w-1)$, then $X$ and its associates are atoms in $\mtwo$.
\end{cor}
\begin{proof}
This is the contrapositive of \cref{wxyzupperbound2}.
\end{proof}

We will consider a subset of bisymmetric matrices with relatively prime adjacent entries,
\[\B^*=\left\{\pmx{x}{y}{y}{x}\in \B: \gcd(x,y)=1\right\}.\]
Note that $\B^*$ is not a subsemigroup as it is not closed under multiplication. Further, not every element of $\B^*$ is an atom; for example, $\pmx{5}{4}{4}{5}\in \B^*$ but is not an atom. However, the condition that $x$ and $y$ are relatively prime is a necessary condition to being an atom.  The upper bound in \cref{wxyzupperbound2} can be improved when the matrix is bisymmetric.

\begin{theorem}\label{cor: upperboundBisym} Let $X=\pmx{x}{y}{y}{x} \in \Bstar$. If $X$ is not an atom then it must satisfy
$$x\leq1+\frac{y^2}{4}\text{ and }y\leq1+\frac{x^2}{4}.$$
Furthermore, these bounds are sharp, i.e., there exists $X$ achieving these bounds.
\end{theorem}

\begin{proof}
By Theorem~\ref{thm:bisym factors}, let $X=AB$ for some nonunits $A=\pmx{a}{b}{b}{a}$ and $B=\pmx{c}{d}{d}{c}$ in $\Bstar$. Note that $a,b,c,d>0$ by Lemma~\ref{LemmaNoZeroes}. Now, $ac+bd=x$, so we let $ac=\dfrac{x}{2}+t$ and $bd=\dfrac{x}{2}-t$ for some real number $t\in\left[1-\dfrac{x}{2},\dfrac{x}{2}-1\right]$. This implies that $ad\leq abcd=\dfrac{x^2}{4}-t^2\leq\dfrac{x^2}{4}$, $c=\dfrac{1}{a}\left(\dfrac{x}{2}+t\right)$, and $b=\dfrac{1}{d}\left(\dfrac{x}{2}-t\right)$. Hence,
$$y=ad+bc=ad+\frac{1}{ad}\left(\frac{x^2}{4}-t^2\right)\leq ad+\frac{1}{ad}\frac{x^2}{4}\leq1+\frac{x^2}{4},$$
where the last inequality is due to $1+\dfrac{x^2}{4}-ad-\dfrac{1}{ad}\dfrac{x^2}{4}=(ad-1)\left(\dfrac{1}{ad}\dfrac{x^2}{4}-1\right)\geq0$. (Note that $ad\leq \frac{x^2}{4}$ implies that $\frac{1}{ad}\geq \frac{4}{x^2}$, and thus $\frac{1}{ad}\frac{x^2}{4}\geq 1$.) The inequality $x\leq1+\dfrac{y^2}{4}$ follows similarly by considering $\pmx{y}{x}{x}{y}=\pmx{b}{a}{a}{b}\pmx{c}{d}{d}{c}$. 

Lastly, we will show that the bounds are sharp. When $x$ is a positive multiple of $4$, then $x$ and $1+\frac{x^2}{4}$ are non-negative integers and relatively prime. Then,
$$\pmx{x}{1+\frac{x^2}{4}}{1+\frac{x^2}{4}}{x}=\pmx{1}{\frac{x}{2}}{\frac{x}{2}}{1}\pmx{\frac{x}{2}}{1}{1}{\frac{x}{2}}\text{ and }\pmx{1+\frac{y^2}{4}}{y}{y}{1+\frac{y^2}{4}}=\pmx{\frac{y}{2}}{1}{1}{\frac{y}{2}}\pmx{\frac{y}{2}}{1}{1}{\frac{y}{2}}$$
achieve these upper bounds.
\end{proof}

After establishing these upper bounds, we can conclude that any matrix which violates these upper bounds is an atom.

\begin{cor}\label{thm:boundBisym} Let $X=\pmx{x}{y}{y}{x} \in \Bstar$. If $\max\{x,y\}> 1+ \frac{(\min\{x,y\})^2}{4}$, then $X$ is an atom in $\mtwo$.
\end{cor}

\begin{proof}
This is a direct result of \cref{cor: upperboundBisym}.
\end{proof}

For a fixed minimum value, this theorem establishes a large class of atoms when the maximum value is sufficiently large and the condition that $\gcd(x,y)=1$ is met. In the next section, we algorithmically search for factors to determine which bisymmetric matrices are atoms up to the upper bound in  \cref{thm:boundBisym}. This gives us a complete description of the atoms and non-atoms, up to the minimum value of our search.

\section{Factor Search Algorithm}\label{sect: algorithm}

Let $X=\pmx{x}{y}{y}{x}\in \Bstar$ be a bisymmetric matrix which factors into $X=AB$, where $A, B \in \mtwo$. Then by \cref{thm:bisym factors}, both factors are bisymmetric, namely $A=\pmx{a}{b}{b}{a}$ and $B=\pmx{c}{d}{d}{c}$, so that
\begin{equation}\label{eq:bisymfactor}
\pmx{x}{y}{y}{x}= \pmx{a}{b}{b}{a}\pmx{c}{d}{d}{c}.
\end{equation}

Note that by \cref{LemmaNoZeroes}, $a,b,c,d>0$ since $X \in \Bstar$. We have created an algorithm to search for possible factorizations. From equation \eqref{eq:bisymfactor}, we know that $x=ac+bd$ and $y=ad+bc$. Rather than search through all possible entries for $A$ and $B$, which would require four possible entries, we can optimize this algorithm by considering possible non-negative integer partitions of $x$. The idea is  to fix $x$ and search possible values of $ac$, which then determines $bd$. We can then check if these values are divisible by $a$ and $b$, respectively, which then determine $c$, $d$, and $y$. Therefore we only need to search through three different variables: $ac$, $a$, and $b$, which minimizes the computational time.

Without loss of generality, we assume that $x<y$. This will not miss any possible factorizations, because if $x>y$, then the associate matrix $\pmx{y}{x}{x}{y}$ will be detected by the algorithm. Note that if $X=AB$ factors into non-negative integer matrices, then multiplying by the exchange matrix $J = \pmx{0}{1}{1}{0}$, the matrix $\pmx{y}{x}{x}{y}=JX=(JA)B$ also factors into non-negative integer matrices.  Note that for $x=1,2,3$, these matrices are already known to be atoms in $\mtwo$ by Theorem \ref{TheoremBaethAtoms}.

Using the above restrictions, we now give an outline of the code. The complete code, implemented in Python, can be found  at \cite{BisymData26}.

\bigskip

\begin{verbatim}
FOR x in minimum values:
    FOR AC in 1 to x-1:
        DEFINE BD=x-AC
        FOR a in 1 to AC:
            IF AC is divisible by a:
                DEFINE c=AC/a
                FOR b in 1 to BD:
                    IF BD is divisible by b:
                        DEFINE d=BD/b
                        DEFINE y=a*d-b*c
                        IF x and y are relatively prime AND x<y
                            Store factorization
\end{verbatim}

The factorizations of matrices with minimum entry up to 4000 can be found at \cite{BisymData26}. For brevity's sake, only the matrices $\pmx{x}{y}{y}{x}$ where $x<y$ are listed, although the associate matrix $\pmx{y}{x}{x}{y}$ will also have a factorization. If a matrix with minimum entry up to 4000 or its associate does not appear in this list, then it is an atom.  

\begin{theorem}\label{thm: algorithm}
Let $X=\pmx{x}{y}{y}{x} \in \Bstar$ with $\min\{x,y\}\leq 4000$. The matrices listed in \cite{BisymData26} (and their associates) are not atoms in $\mtwo$. If $X$ and its associate matrix do not appear in \cite{BisymData26}, then $X$ is an atom.
\end{theorem}

In Appendix \ref{sec:Appendix}, we list the factorable bisymmetric matrices with relatively prime entries and minimum up to 42. The bisymmetric non-atoms with minimum up to 12 and their factorizations appear in Table \ref{table: main}. The bisymmetric matrices from minimum 13 to 42 that are not atoms are listed in Table \ref{table: extended} (without the factorization). For brevity, only the matrices with $x<y$ are listed. All other bisymmetric matrices with relatively prime entries and minimum entry up to 42 are atoms.

\bibliography{References.bib}

\appendix

\section{Tables of Factorizations}\label{sec:Appendix}

\begin{table}[H]
\caption{Bisymmetric Matrices that Factor and satisfy $\gcd(x,y)=1$}
$\begin{array}{|c|c|c|}
    \hline
    \text{Minimum value}& \text{Non-atoms satisfying }\gcd(x,y)=1& \text{Factors (up to associates)}\\
    \hline
   x=4& \pmx{4}{5}{5}{4} & \pmx{1}{2}{2}{1} \pmx{2}{1}{1}{2}\\
    x=5 &\pmx{5}{7}{7}{5} & \pmx{1}{2}{2}{1} \pmx{3}{1}{1}{3}\\
    x=6 & \text{none} &\\
    x=7 & \pmx{7}{8}{8}{7} & \pmx{1}{2}{2}{1}\pmx{3}{2}{2}{3}\\
        & \pmx{7}{11}{11}{7}& \pmx{1}{2}{2}{1}\pmx{5}{1}{1}{5}\\
        & \pmx{7}{13}{13}{7}& \pmx{1}{3}{3}{1}\pmx{4}{1}{1}{4}\\
    x=8 & \pmx{8}{13}{13}{8} & \pmx{1}{2}{2}{1}\pmx{6}{1}{1}{6}\\
        & \pmx{8}{17}{17}{8}&\pmx{1}{4}{4}{1}\pmx{4}{1}{1}{4}\\
    x=9 & \pmx{9}{11}{11}{9} & \pmx{1}{3}{3}{1}\pmx{3}{2}{2}{3}\\
        & \pmx{9}{19}{19}{9} & \pmx{1}{3}{3}{1}\pmx{6}{1}{1}{6}\\
    x=10& \pmx{10}{11}{11}{10} & \pmx{1}{2}{2}{1}\pmx{4}{3}{3}{4}\\
        &\pmx{10}{17}{17}{10}& \pmx{1}{2}{2}{1} \pmx{8}{1}{1}{8}\\
    x=11& \pmx{11}{13}{13}{11}&\pmx{1}{2}{2}{1}\pmx{5}{3}{3}{5}\\
        & \pmx{11}{14}{14}{11}&\pmx{1}{4}{4}{1}\pmx{3}{2}{2}{3}\\
        & \pmx{11}{16}{16}{11}&\pmx{1}{2}{2}{1}\pmx{7}{2}{2}{7}\\
        &\pmx{11}{17}{17}{11} & \pmx{1}{3}{3}{1}\pmx{5}{2}{2}{5}\\
        &\pmx{11}{19}{19}{11} & \pmx{1}{2}{2}{1}\pmx{9}{1}{1}{9}\\
        &\pmx{11}{25}{25}{11} & \pmx{1}{3}{3}{1}\pmx{8}{1}{1}{8}\\
        & \pmx{11}{29}{29}{11} &\pmx{1}{4}{4}{1}\pmx{7}{1}{1}{7}\\
        & \pmx{11}{31}{31}{11} & \pmx{1}{5}{5}{1}\pmx{6}{1}{1}{6}\\
    x=12& \pmx{12}{13}{13}{12} & \pmx{2}{3}{3}{2}\pmx{3}{2}{2}{3}\\
        & \pmx{12}{37}{37}{12}& \pmx{1}{6}{6}{1} \pmx{6}{1}{1}{6}\\

        \hline
\end{array}$
\label{table: main}
\end{table}

\newpage

\begin{table}[H]

\caption{$x,y$ Values of Bisymmetric Matrices that Factor and satisfy $\gcd(x,y)=1$}
$\begin{array}{|c|l|}
\hline
x&y\\
\hline
13 & 14, 15, 17,20,22,23,31,37,41,43\\
14 & 19,25,41\\
15 & 29, 37\\
16 & 17, 19, 23, 29, 49, 61, 65\\
17 & 18, 19, 22, 23, 25, 27, 28, 31, 32, 35, 37, 38, 43,  53, 61, 67, 71, 73\\
18 & 73\\
19 & 20, 21, 23, 25, 26, 29, 31, 32, 33, 35, 37, 41, 44, 46, 47, 49, 61, 71, 79, 85, 89, 91\\
20 & 29, 31, 37, 97, 101\\
21 & 23, 29, 31, 47, 55, 109\\
22 & 23, 27, 29, 35, 41, 43, 73, 97, 113\\
23 & 25, 26, 27, 28, 29, 31, 32, 33, 34, 37, 40, 43, 45, 47, 53, 58, 61, 62, 65, 67, 68, 77, 91, 103, 113, 121, 127, 131, 133\\
24 & 25, 31, 109, 145\\
25 & 26, 27, 29, 31, 32, 38, 41, 43, 44, 47, 51, 52, 53, 59, 67, 74, 77, 79, 101, 127, 137, 151, 157\\
26 & 29, 31, 37, 43, 49, 51, 59, 89, 121, 145, 161\\
27 & 28, 29, 38, 41, 49, 65, 73, 83, 92, 127, 163, 181\\
28 & 29, 37, 41, 47, 53, 67, 97, 181, 193, 197\\
29 & 31, 34, 36, 37, 39, 40, 41, 43, 46, 47, 49, 52, 55, 56, 59, 62, 63, 69, 71, 73, 79, 86, 92, 97, 101, 104, 106, 107, 121,\\
&139, 155, 169, 181, 191, 199, 205, 209, 211\\
30 & 217\\
31 & 32, 34, 35, 37, 38, 39, 41, 44, 45, 46, 47, 49, 50, 53, 56, 59, 61, 64, 67, 69, 73, 77, 79, 81, 83, 85, 86, 94, 101, 107,\\
&109, 112, 116, 119, 121, 122, 131, 151, 169, 185, 199, 211, 221, 229, 235, 239, 241\\
32 & 33, 37, 43, 49, 53, 55, 59, 61, 67, 83, 87, 113, 157, 193, 221, 241, 253, 257\\
33 & 35, 37, 43, 47, 58, 59, 67, 83, 91, 128, 137, 163, 217, 271\\
34 & 35, 41, 43, 47, 53, 59, 61, 65, 83, 91, 99, 121, 169, 209, 241, 265, 281\\
35 & 37, 41, 43, 46, 52, 53, 57, 58, 61, 64, 67, 73, 79, 81, 89, 97, 101, 103, 127, 134, 149, 151, 152, 197, 251, 277, 307\\
36 & 41, 49, 85, 181, 289, 325\\
37 & 38, 39, 40, 41, 43, 44, 47, 48, 50, 51, 53, 55, 56, 58, 59, 61, 62, 63, 65, 67, 68, 71, 73, 79, 82, 87, 88, 89, 93, 95,\\
&103, 107, 113, 115, 117, 118, 128, 133, 137, 145, 152, 158, 161, 163, 167, 170, 172, 173, 187, 211, 233, 253, 271,\\
&287, 301, 313, 323, 331, 337, 341, 343\\
38 & 39, 43, 47, 49, 53, 55, 61, 67, 73, 77, 83, 107, 115, 123, 137, 193, 241, 281, 313, 337, 353\\
39 & 41, 46, 49, 53, 56, 59, 61, 77, 85, 94, 101, 109, 137, 164, 191, 199, 271, 361, 379\\
40 & 41, 47, 51, 53, 59, 71, 77, 79, 103, 131, 257, 301, 337, 397, 401\\
41 & 43, 44, 46, 47, 49, 50, 51, 52, 54, 55, 58, 59, 61, 63, 64, 67, 69, 70, 71, 73, 74, 75, 76, 79, 83, 85, 89, 91, 92, 95,\\
&99, 104, 106, 107, 109, 113, 115, 119, 121, 129, 133, 134, 139, 141, 143, 146, 149, 157, 167, 176, 181, 184, 191,\\
&197, 202, 206, 209, 211, 212, 239, 265, 289, 311, 331, 349, 365, 379, 391, 401, 409, 415, 419, 421
\\
42 & 43, 53, 361, 433\\
\hline
\end{array}$
\label{table: extended}
\end{table}

\end{document}